\documentclass[a4paper]{article}%
\usepackage[dvips]{graphicx}
\usepackage{amsmath}
\usepackage{amsfonts}
\usepackage{amssymb}
\usepackage{latexsym}
\usepackage{makeidx}
\usepackage[noBBpl]{mathpazo}
\usepackage{indentfirst}%
\providecommand{\U}[1]{\protect\rule{.1in}{.1in}}
\newtheorem{theorem}{Theorem}

\newtheorem{corollary}[theorem]{Corollary}

\newtheorem{lemma}[theorem]{Lemma}

\newtheorem{proposition}[theorem]{Proposition}
\newtheorem{remark}[theorem]{Remark}

\newenvironment{proof}[1][Proof]{\noindent\textbf{#1.} }{\hfill$\Box$\vspace{.5cm}}
\begin{document}

\title{\noindent Absolute continuity of the best Sobolev constant of a bounded domain }
\author{Grey Ercole{\small \textbf{ }}\thanks{E-mail: grey@mat.ufmg.br. The author was
supported by FAPEMIG and CNPq, Brazil.}\\{\small \textit{Departamento de Matem\'{a}tica - ICEx, Universidade Federal de
Minas Gerais,}}\\{\small \textit{Av. Ant\^{o}nio Carlos 6627, Caixa Postal 702, 30161-970, Belo
Horizonte, MG, Brazil }} }
\maketitle

\begin{abstract}
\noindent Let $\lambda_{q}:=\inf\left\{  \left\Vert \nabla u\right\Vert
_{L^{p}(\Omega)}^{p}/\left\Vert u\right\Vert _{L^{q}(\Omega)}^{p}:u\in
W_{0}^{1,p}(\Omega)\setminus\{0\}\right\}  $, where $\Omega$ is a bounded and
smooth domain of $\mathbb{R}^{N},$ $1<p<N$ and $1\leq q\leq p^{\star}%
:=\frac{Np}{N-p}.$ We prove that the function $q\mapsto\lambda_{q}$ is
absolutely continuous in the closed interval $[1,p^{\star}].$ \newline



\end{abstract}

\section{Introduction.}

\noindent Let $\Omega$ be a bounded and smooth domain of Euclidean space
$\mathbb{R}^{N},$ $N\geq2,$ and let $1<p<N.$ For each $1\leq q\leq p^{\star
}:=\dfrac{Np}{N-p}$, let $\mathcal{R}_{q}:W_{0}^{1,p}(\Omega)\setminus
\{0\}\longrightarrow\mathbb{R}$ be the Rayleigh quotient associated with the
Sobolev immersion $W_{0}^{1,p}(\Omega)\hookrightarrow L^{q}(\Omega)$. That
is,
\[
\mathcal{R}_{q}(u):=\left(  {\int_{\Omega}}\left\vert \nabla u\right\vert
^{p}dx\right)  \left(  {\int_{\Omega}}\left\vert u\right\vert ^{q}dx\right)
^{-\frac{q}{p}}=\frac{\left\Vert \nabla u\right\Vert _{p}^{p}}{\left\Vert
u\right\Vert _{q}^{p}}%
\]
where $\left\Vert \cdot\right\Vert _{s}:=\left(  {\int_{\Omega}}\left\vert
\cdot\right\vert ^{s}dx\right)  ^{\frac{1}{s}}$ denotes the usual norm of
$L^{s}(\Omega).$

It is well-known that the immersion $W_{0}^{1,p}(\Omega)\hookrightarrow
L^{q}(\Omega)$ is continuous if $1\leq q\leq p^{\star}$ and compact if $1\leq
q<p^{\star}.$ Hence, there exist
\begin{equation}
\lambda_{q}:=\inf\left\{  \mathcal{R}_{q}(u):u\in W_{0}^{1,p}(\Omega
)\setminus\{0\}\right\}  ,\ 1\leq q\leq p^{\star} \label{lambdaq}%
\end{equation}
and $w_{q}\in W_{0}^{1,p}(\Omega)\setminus\{0\}$ such that
\begin{equation}
\mathcal{R}_{q}(w_{q})=\lambda_{q},\text{ \ }1\leq q<p^{\star}. \label{wq}%
\end{equation}
Since $\mathcal{R}_{q}$ is homogeneous of degree zero the extremal function
$w_{q}$ for the Rayleigh quotient can be chosen such that $\left\Vert
w_{q}\right\Vert _{q}=1.$

It is straightforward to verify that such a normalized extremal $w_{q}$ is a
weak solution of the Dirichlet problem%
\begin{equation}
\left\{
\begin{array}
[c]{rrll}%
-\Delta_{p}u & = & \lambda_{q}\left\vert u\right\vert ^{q-2}u & \text{in
}\Omega\\
u & = & 0 & \text{on }\partial\Omega,
\end{array}
\right.  \label{diricq}%
\end{equation}
for the $p$-Laplacian operator $\Delta_{p}u:=\operatorname{div}(\left\vert
\nabla u\right\vert ^{p-2}\nabla u).$ Hence, classical results imply that
$w_{q}$ can still be chosen to be positive in $\Omega$ and that $w_{q}\in
C^{1,\alpha}(\overline{\Omega})$ for some $0<\alpha<1.$

In the case $q=p$, the constant $\lambda_{p}$ is the well-known first
eigenvalue of the Dirichlet $p$-Laplacian and $w_{p}$ is the correspondent
eigenfunction $L^{p}$-normalized.

If $q=1$ the pair $(\lambda_{1},w_{1})$ is obtained from the \textit{Torsional
Creep Problem}:%
\begin{equation}
\left\{
\begin{array}
[c]{rrll}%
-\Delta_{p}u & = & 1 & \text{in }\Omega\\
u & = & 0 & \text{on }\partial\Omega.
\end{array}
\right.  \label{torsional}%
\end{equation}
In fact, if $\phi_{p}$ is the \textit{torsion function} of $\Omega,$ that is,
the solution of (\ref{torsional}), then it easy to check that the only
positive weak solution of (\ref{diricq}) with $q=1$ is $\lambda_{1}^{\frac
{1}{p-1}}\phi_{p}.$ Thus, $w_{1}=\lambda_{1}^{\frac{1}{p-1}}\phi_{p}$ and
since $\left\Vert w_{1}\right\Vert _{1}=1$ one has
\begin{equation}
\lambda_{1}=\frac{1}{\left\Vert \phi_{p}\right\Vert _{1}^{p-1}}\text{ \ and
\ }w_{1}=\frac{\phi_{p}}{\left\Vert \phi_{p}\right\Vert _{1}}.
\label{c1=torsion}%
\end{equation}
In the particular case where $\Omega=B_{R}(x_{0}),$ the ball of radius $R>0$
centered at $x_{0}\in\mathbb{R}^{N},$ the torsion function is explicitly given
by $\phi_{p}(x)=\Phi_{p}(\left\vert x-x_{0}\right\vert )$ where
\[
\Phi_{p}\left(  r\right)  :=\frac{p-1}{p}N^{-\frac{1}{p-1}}\left(  R^{\frac
{p}{p-1}}-r^{\frac{p}{p-1}}\right)  ,\text{ \ }0\leq r\leq R.
\]
Hence, for $\Omega=B_{R}(x_{0})$ one obtains%
\begin{equation}
\lambda_{1}=\left[  \frac{p+N(p-1)}{\omega_{N}(p-1)}\right]  ^{p-1}\frac
{N}{R^{(p^{\ast}-1)(N-p)}} \label{c1ball}%
\end{equation}
and%
\[
w_{1}(x)=\frac{p+N(p-1)}{p\omega_{N}R^{N}}\left(  1-\left(  \left\vert
x-x_{0}\right\vert /R\right)  ^{\frac{p}{p-1}}\right)
\]
where $\omega_{N}$ is the $N$ dimensional Lebesgue volume of the unit ball
$B_{1}(0).$ (More properties of the torsion function and some of its
applications are given in \cite{BE, Kawohl}.)

In the critical case $q=p^{\star}$ extremals for the Rayleigh quotient exist
if the domain is the whole Euclidean space $\mathbb{R}^{N}.$ In fact, in
$\mathbb{R}^{N}$ one has the \textit{Sobolev Inequality}
\begin{equation}
\mathcal{S}_{p,N}\left\Vert u\right\Vert _{L^{p^{\star}}(\mathbb{R}^{N})}%
\leq\left\Vert \nabla u\right\Vert _{L^{p}(\mathbb{R}^{N})}\text{ \ for all
}u\in W^{1,p}(\mathbb{R}^{N}) \label{Sobineq}%
\end{equation}
where (see \cite{Auban, Talenti}):
\begin{equation}
\mathcal{S}_{p,N}:=\sqrt{\pi}N^{\frac{1}{p}}\left(  \frac{N-p}{p-1}\right)
^{\frac{p-1}{p}}\left(  \frac{\Gamma(N/p)\Gamma(1+N-N/p)}{\Gamma
(1+N/2)\Gamma(N)}\right)  ^{\frac{1}{N}} \label{SNp}%
\end{equation}
and $\Gamma(t)=\int_{0}^{\infty}s^{t-1}e^{-s}ds$ is the Gamma Function. The
\textit{Sobolev constant} $\mathcal{S}_{p,N}$ is optimal and achieved,
necessarily, by radially symmetric functions of the form (see \cite{Talenti}%
):
\begin{equation}
w(x)=a\left(  1+b\left\vert x-x_{0}\right\vert ^{\frac{p}{p-1}}\right)
^{-\frac{N-p}{p}} \label{Nextremal}%
\end{equation}
for any $a\neq0,$ $b>0$ and $x_{0}\in\mathbb{R}^{N}.$

A remarkable fact is that for any domain $\Omega$ (open, but non-necessarily
bounded) the Sobolev constant $\mathcal{S}_{p,N}$ is still sharp with respect
to the inequality (\ref{Sobineq}), that is:
\begin{equation}
\mathcal{S}_{p,N}^{p}=\lambda_{p^{\star}}:=\inf\left\{  \mathcal{R}_{p^{\star
}}(u):u\in W_{0}^{1,p}(\Omega)\setminus\{0\}\right\}  . \label{lambda*}%
\end{equation}
This property of the critical case $q=p^{\star}$ may be easily verified by
using a simple scaling argument. As a consequence, in this critical case, the
only domain $\Omega$ whose the Rayleigh quotient has an extremal is
$\mathbb{R}^{N}.$ Indeed, if $w\in W_{0}^{1,p}(\Omega)/\{0\}$ is an extremal
for the Rayleigh quotient in $\Omega,$ then (by extending $w$ to zero out of
$\Omega$) $w$ is also an extremal for the Rayleigh quotient in $\mathbb{R}%
^{N}.$ This implies that $w$ must have an expression as in (\ref{Nextremal})
and hence its support must be the whole space $\mathbb{R}^{N},$ forcing thus
the equality $\Omega=\mathbb{R}^{N}.$

In this paper we are concerned with the behavior of $\lambda_{q}$ with respect
to $q\in\lbrack1,p^{\star}].$ Thus, we investigate the function $q\mapsto
\lambda_{q}$ defined by (\ref{lambdaq}). We prove that this function is of
bounded variation in $[1,p^{\star}],$ Lipschitz continuous in any closed
interval of the form $[1,p^{\star}-\epsilon]$ for $\epsilon>0$, and
left-continuous at $q=p^{\star}.$ These combined results imply that
$\lambda_{q}$ is absolute continuous on $[1,p^{\star}].$

Up to our knowledge, the only result about the continuity of the function
$q\mapsto\lambda_{q}$ is given in \cite[Thm 2.1]{Huang}, where the author
proves the continuity of this function in the open interval $(1,p)$ and the
lower semi-continuity in the open interval $(p,p^{\star}).$

Besides the theoretical aspects, our results are also important for the
computational approach of the Sobolev constants $\lambda_{q},$ since these
constants or the correspondent extremals are not explicitly known in general,
even for simple bounded domains. For recent numerical approaches related to
Sobolev type constants we refer to \cite{Antonietti,Caboussat}.

This paper is organized as follows. In Section 2 we derive a formula that
describes the dependence of $\mathcal{R}_{q}$ with respect to $q$ and obtain,
in consequence, the bounded variation of the function $q\mapsto\lambda_{q}$ in
the closed interval $[1,p^{\star}]$ and also the left-continuity of this
function at $q=p^{\star}.$ Still in Section 2 we obtain a upper bound for
$\mathcal{S}_{p,N}$ (see (\ref{boundSNp})) and we also show that for $1\leq
q<p^{\star}$ the Sobolev constant $\lambda_{q}$ of bounded domains $\Omega$
tends to zero when these domains tend to $\mathbb{R}^{N}.$

By applying set level techniques, we deduce in Section 3 some estimates for
$w_{q}$ and in Section 4 we combine these estimates with the formula derived
in Section 2 to prove the Lipschitz continuity of the function $q\mapsto
\lambda_{q}$ in each closed interval of the form $[1,p^{\star}-\epsilon].$ Our
results are in fact proved for the function $q\mapsto\left\vert \Omega
\right\vert ^{\frac{p}{q}}\lambda_{q}$ where $\left\vert \Omega\right\vert $
denotes the $N$-dimensional Lebesgue volume of $\Omega.$ But, of course, they
are automatically transferred to the function $q\mapsto\lambda_{q}.$

\section{Bounded variation and left-continuity}

We first describe the dependence of the Rayleigh quotient $\mathcal{R}_{q}(u)$
with respect to the parameter $q.$

\begin{lemma}
\label{formula}Let $0\not \equiv u\in W_{0}^{1,p}(\Omega)\cap L^{\infty
}(\Omega).$ Then, for each $1\leq s_{1}<s_{2}\leq p^{\ast}$ one has%
\begin{equation}
\left\vert \Omega\right\vert ^{\frac{p}{s_{1}}}\mathcal{R}_{s_{1}%
}(u)=\left\vert \Omega\right\vert ^{\frac{p}{s_{2}}}\mathcal{R}_{s_{2}}%
(u)\exp\left(  p\int_{s_{1}}^{s_{2}}\frac{K(t,u)}{t^{2}}dt\right)  \label{edo}%
\end{equation}
where
\begin{equation}
K(t,u):=\frac{\int_{\Omega}\left\vert u\right\vert ^{t}\ln\left\vert
u\right\vert ^{t}dx}{\left\Vert u\right\Vert _{t}^{t}}+\ln\left(  \left\vert
\Omega\right\vert \left\Vert u\right\Vert _{t}^{-t}\right)  \geq0.
\label{Klessthan}%
\end{equation}

\end{lemma}

Before proving Lemma \ref{formula} let us make a technical remark related to
the assumptions of this lemma. If $u\in W_{0}^{1,p}(\Omega)$ and $1\leq
t<p^{\star},$ then
\begin{align*}
\int_{\Omega}\left\vert u\right\vert ^{t}\left\vert \ln\left\vert u\right\vert
^{t}\right\vert dx  &  =\int_{\left\vert u\right\vert <1}\left\vert
u\right\vert ^{t}\left\vert \ln\left\vert u\right\vert ^{t}\right\vert
dx+t\int_{\left\vert u\right\vert \geq1}\left\vert u\right\vert ^{t}%
\ln\left\vert u\right\vert dx\\
&  \leq\frac{\left\vert \Omega\right\vert }{e}+t\int_{\left\vert u\right\vert
\geq1}\left\vert u\right\vert ^{t}\frac{\left\vert u\right\vert ^{p^{\star}%
-t}}{e(p^{\star}-t)}dx\leq\frac{\left\vert \Omega\right\vert }{e}%
+\frac{p^{\star}\left\Vert u\right\Vert _{p^{\star}}^{p^{\star}}}{e(p^{\star
}-t)}<\infty.
\end{align*}
However, we were not able to determine the finiteness of the integral
$\int_{\Omega}\left\vert u\right\vert ^{p^{\star}}\left\vert \ln\left\vert
u\right\vert \right\vert dx$ without assuming that $u\in L^{\infty}(\Omega).$
Fortunately, the assumption $W_{0}^{1,p}(\Omega)\cap L^{\infty}(\Omega)$ will
be sufficient to our purposes in this paper.

\begin{proof}
[Proof of Lemma \ref{formula}]We firstly note that%
\begin{align*}
\frac{d}{dq}\ln\left(  \frac{\left\vert \Omega\right\vert ^{\frac{1}{q}}%
}{\left\Vert u\right\Vert _{q}}\right)   &  =\frac{d}{dq}\frac{\ln\left\vert
\Omega\right\vert }{q}-\frac{d}{dq}\left(  \frac{1}{q}\ln\int_{\Omega
}\left\vert u\right\vert ^{q}dx\right) \\
&  =-\frac{1}{q^{2}}\ln\left\vert \Omega\right\vert -\frac{1}{q^{2}}\left[
-\ln\left\Vert u\right\Vert _{q}^{q}+\frac{\int_{\Omega}\left\vert
u\right\vert ^{q}\ln\left\vert u\right\vert ^{q}dx}{\left\Vert u\right\Vert
_{q}^{q}}\right]  =-\frac{K(q,u)}{q^{2}}.
\end{align*}
Thus, integration on the interval $[s_{1},s_{2}]$ gives
\[
\frac{\left\vert \Omega\right\vert ^{\frac{1}{s_{2}}}}{\left\Vert u\right\Vert
_{s_{2}}}=\frac{\left\vert \Omega\right\vert ^{\frac{1}{s_{1}}}}{\left\Vert
u\right\Vert _{s_{1}}}\exp\left(  -\int_{s_{1}}^{s_{2}}\frac{K(t,u)}{t^{2}%
}dt\right)
\]
from what (\ref{edo}) follows easily.

Since the continuous function $h:[0,+\infty)\longrightarrow\mathbb{R}$ defined
by $h(\xi):=\xi\ln\xi$, if $\xi>0$, and $h(0)=0$ is convex, it follows from
Jensen's inequality that%
\[
h\left(  \left\vert \Omega\right\vert ^{-1}\int_{\Omega}\left\vert
u\right\vert ^{t}dx\right)  \leq\left\vert \Omega\right\vert ^{-1}\int
_{\Omega}h(\left\vert u\right\vert ^{t})dx,
\]
thus yielding%
\[
\left\Vert u\right\Vert _{t}^{t}\ln\left(  \left\vert \Omega\right\vert
^{-1}\left\Vert u\right\Vert _{t}^{t}\right)  \leq\int_{\Omega}\left\vert
u\right\vert ^{t}\ln\left\vert u\right\vert ^{t}dx,
\]
from what follows that $K(t,u)$ defined in (\ref{Klessthan}) is nonnegative.
\end{proof}

\begin{proposition}
\label{monot}The function $q\mapsto\left\vert \Omega\right\vert ^{\frac{p}{q}%
}\lambda_{q}$ is strictly decreasing in $[1,p^{\star}].$
\end{proposition}

\begin{proof}
Let $1\leq s_{1}<s_{2}\leq p^{\star}$ and $w_{s_{1}}\in W_{0}^{1,p}%
(\Omega)\cap C^{1,\alpha}(\overline{\Omega})$ the positive and $L^{s_{1}}%
$-normalized extremal of the Rayleigh quotient $\mathcal{R}_{s_{1}}.$ Note
from the definition of $w_{s_{1}}$ that%
\begin{equation}
-\Delta_{p}w_{s_{1}}=\lambda_{s_{1}}w_{s_{1}}^{s_{1}-1}\text{ \ in \ }\Omega.
\label{aux1}%
\end{equation}

It follows from Lemma \ref{formula} that%
\[
\left\vert \Omega\right\vert ^{\frac{p}{s_{1}}}\lambda_{s_{1}}=\left\vert
\Omega\right\vert ^{\frac{p}{s_{2}}}\mathcal{R}_{s_{2}}(w_{s_{1}})\exp\left(
p\int_{s_{1}}^{s_{2}}\frac{K(t,w_{s_{1}})}{t^{2}}dt\right)  \geq\left\vert
\Omega\right\vert ^{\frac{p}{s_{2}}}\mathcal{R}_{s_{2}}(w_{s_{1}})>\left\vert
\Omega\right\vert ^{\frac{p}{s_{2}}}\lambda_{s_{2}}%
\]
since $\mathcal{R}_{s_{1}}(w_{s_{1}})=\lambda_{s_{1}},$ $K(t,w_{s_{1}})\geq0$
and $\mathcal{R}_{s_{2}}(w_{s_{1}})>\lambda_{s_{2}}.$ We need only to
guarantee the strictness of the last inequality. Obviously, if $s_{2}%
=p^{\star}$ the inequality is really strict because the Rayleigh quotient
$\mathcal{R}_{p\star}$ does not reach a minimum value. Thus, let us suppose
that $\lambda_{s_{2}}=\mathcal{R}_{s_{2}}(w_{s_{1}})$ for $s_{2}<p^{\star}.$
Then
\[
-\Delta_{p}(w_{s_{1}}/\left\Vert w_{s_{1}}\right\Vert _{s_{2}})=\lambda
_{s_{2}}(w_{s_{1}}/\left\Vert w_{s_{1}}\right\Vert _{s_{2}})^{s_{2}-1}%
\]
and hence the $(p-1)$-homogeneity of the operator $\Delta_{p}$ yields
\begin{equation}
-\Delta_{p}w_{s_{1}}=\lambda_{s_{2}}\left\Vert w_{s_{1}}\right\Vert _{s_{2}%
}^{p-s_{2}}w_{s_{1}}^{s_{2}-1}\text{ \ in \ }\Omega.\label{aux2}%
\end{equation}
The combining of (\ref{aux1}) with (\ref{aux2}) produces
\[
w_{s_{1}}\equiv\left(  \frac{\lambda_{s_{1}}\left\Vert w_{s_{1}}\right\Vert
_{s_{2}}^{s_{2}-p}}{\lambda_{s_{2}}}\right)  ^{\frac{1}{s_{2}-s_{1}}}\text{
\ in \ }\Omega.
\]
Since the only constant function in $W_{0}^{1,p}(\Omega)$ is the null function
we arrive at the contradiction $0\equiv w_{s_{1}}>0$ in $\Omega.$

Thus, we have concluded that $\left\vert \Omega\right\vert ^{\frac{p}{s_{1}}%
}\lambda_{s_{1}}>\left\vert \Omega\right\vert ^{\frac{p}{s_{2}}}\lambda
_{s_{2}}$ for $1\leq s_{1}<s_{2}\leq p^{\star}.$
\end{proof}

The following corollary is immediate after writing $\lambda_{q}$ as a product
of two monotonic functions: $\lambda_{q}=\left\vert \Omega\right\vert
^{-\frac{p}{q}}(\left\vert \Omega\right\vert ^{\frac{p}{q}}\lambda_{q}).$

\begin{corollary}
\label{bv}The function $q\mapsto\lambda_{q}$ is of bounded variation in
$[1,p^{\star}].$
\end{corollary}

Another consequence of Proposition \ref{monot} is that for each $1\leq
q<p^{\ast}$ the Sobolev constant $\lambda_{q}$ of a bounded domain $\Omega$
tends to zero as $\Omega\nearrow\mathbb{R}^{N}.$ In fact, this asymptotic
behavior follows from the following corollary.

\begin{corollary}
Let $B_{R}(x_{0})\subset\mathbb{R}^{N}$ denote the ball centered at $x_{0}$
and with radius $R$ and let%
\[
\lambda_{q}(R):=\min\left\{  \frac{\left\Vert \nabla u\right\Vert
_{L^{p}(B_{R}(x_{0}))}^{p}}{\left\Vert u\right\Vert _{L^{q}(B_{R}(x_{0}))}%
^{p}}:u\in W_{0}^{1,p}(B_{R}(x_{0}))/\{0\}\right\}  ,\hspace{0.5cm}1\leq
q<p^{\star}.
\]
Then
\begin{equation}
\lambda_{q}(R)\rightarrow0\text{ \ \ as \ }R\rightarrow\infty. \label{Rinfty}%
\end{equation}

\end{corollary}

\begin{proof}
It follows from Proposition \ref{monot} that
\[
\lambda_{q}(R)\leq\lambda_{1}(R)(\omega_{N}R^{N})^{p(1-\frac{1}{q})}%
\]
where, as before, $\omega_{N}=\left\vert B_{1}(0)\right\vert .$

Now, replacing $\lambda_{1}(R)$ by its expression (\ref{c1ball}) we obtain
\begin{equation}
\lambda_{q}(R)\leq\left[  \frac{p+N(p-1)}{\omega_{N}(p-1)}\right]  ^{p-1}%
\frac{N(\omega_{N})^{p(1-\frac{1}{q})}}{R^{(N-p)(\frac{p^{\star}}{q}-1)}%
},\label{lambdaR}%
\end{equation}
yielding (\ref{Rinfty}).
\end{proof}

\begin{remark}
Since $\lambda_{p^{\ast}}(R)\equiv\mathcal{S}_{N,p}^{p},$ $\omega_{N}%
=\pi^{N/2}/\Gamma(1+N/2)$ and $\frac{1}{p}-\frac{1}{p^{\star}}=\frac{1}{N},$
by making $q=p^{\ast}$ in (\ref{lambdaR}) we obtain the following upper bound
for $\mathcal{S}_{N,p}$ with is quite comparable with the expression
(\ref{SNp}):
\begin{equation}
\mathcal{S}_{N,p}\leq\sqrt{\pi}N^{\frac{1}{p}}\left(  \frac{N-p}{p-1}\right)
^{\frac{p-1}{p}}\frac{(p^{\ast}-1)^{\frac{p-1}{p}}}{\Gamma(1+N/2)^{\frac{1}%
{N}}}. \label{boundSNp}%
\end{equation}

\end{remark}

We now prove the left-continuity of the function $q\mapsto\lambda_{q}$ in the
interval $(1,p^{\star}].$ Hence, as a particular case we obtain
\begin{equation}
\lim_{q\rightarrow(p^{\star})^{-}}\lambda_{q}=\lambda_{p^{\star}}%
\quad(=\mathcal{S}_{p,N}^{p}). \label{leftp*}%
\end{equation}

\begin{theorem}
\label{T1}For each $q\in(1,p^{\star}]$ it holds $\lim\limits_{s\rightarrow
q^{-}}\lambda_{s}=\lambda_{q}.$
\end{theorem}

\begin{proof}
Let us fix $s<q$ and $u\in C_{c}^{\infty}(\Omega)\setminus\{0\}.$ If follows
from Lemma \ref{formula} and Proposition \ref{monot} that%
\begin{equation}
\left\vert \Omega\right\vert ^{\frac{p}{q}}\lambda_{q}<\left\vert
\Omega\right\vert ^{\frac{p}{s}}\lambda_{s}\leq\left\vert \Omega\right\vert
^{\frac{p}{s}}\mathcal{R}_{s}(u)=\left\vert \Omega\right\vert ^{\frac{p}{q}%
}\mathcal{R}_{q}(u)\exp\left(  p\int_{s}^{q}\frac{K(t,u)}{t^{2}}dt\right)
.\label{cqcscq}%
\end{equation}

For $s\leq t\leq q$ H\"{o}lder's inequality implies that%
\[
\left\vert \Omega\right\vert ^{-\frac{1}{s}}\left\Vert u\right\Vert _{s}%
\leq\left\vert \Omega\right\vert ^{-\frac{1}{t}}\left\Vert u\right\Vert
_{t}\leq\left\vert \Omega\right\vert ^{-\frac{1}{q}}\left\Vert u\right\Vert
_{q}.
\]
$\ $Hence, since $\left\vert \Omega\right\vert ^{-\frac{1}{s}}\left\Vert
u\right\Vert _{s}\rightarrow\left\vert \Omega\right\vert ^{-\frac{1}{q}%
}\left\Vert u\right\Vert _{q}$ as $s\rightarrow q$ we obtain
\[
\frac{\left\vert \Omega\right\vert ^{-\frac{1}{q}}\left\Vert u\right\Vert
_{q}}{2}\leq\left\vert \Omega\right\vert ^{-\frac{1}{s}}\left\Vert
u\right\Vert _{s}\leq\left\vert \Omega\right\vert ^{-\frac{1}{t}}\left\Vert
u\right\Vert _{t}\leq\left\vert \Omega\right\vert ^{-\frac{1}{q}}\left\Vert
u\right\Vert _{q}%
\]
for $s\leq t\leq q\ $with $s$ sufficiently close to $q.$

It follows from these estimates that%
\begin{align*}
K(t,u) &  =\frac{t\int_{\Omega}\left\vert u\right\vert ^{t}\ln\left\vert
u\right\vert dx}{\left\Vert u\right\Vert _{t}^{t}}+t\ln\left(  \frac
{\left\vert \Omega\right\vert ^{\frac{1}{t}}}{\left\Vert u\right\Vert _{t}%
}\right)  \\
&  \leq\frac{t\ln\left\Vert u\right\Vert _{\infty}\int_{\Omega}\left\vert
u\right\vert ^{t}dx}{\left\Vert u\right\Vert _{t}^{t}}+t\ln\left(
\frac{2\left\vert \Omega\right\vert ^{\frac{1}{q}}}{\left\Vert u\right\Vert
_{q}}\right)  =t\ln\left(  \frac{2\left\vert \Omega\right\vert ^{\frac{1}{q}%
}\left\Vert u\right\Vert _{\infty}}{\left\Vert u\right\Vert _{q}}\right)
=:tM_{q}(u).
\end{align*}
Therefore,
\[
\exp\left(  p\int_{s}^{q}\frac{K(t,u)}{t^{2}}dt\right)  \leq\exp\left(
pM_{q}(u)\ln(\frac{q}{s})\right)  =\left(  \frac{q}{s}\right)  ^{pM_{q}(u)}%
\]
and (\ref{cqcscq}) yields%
\[
\left\vert \Omega\right\vert ^{\frac{p}{q}}\lambda_{q}<\left\vert
\Omega\right\vert ^{\frac{p}{s}}\lambda_{s}\leq\left\vert \Omega\right\vert
^{\frac{p}{q}}\mathcal{R}_{q}(u)\left(  \frac{q}{s}\right)  ^{pM_{q}(u)}.
\]

By making $s\rightarrow q^{-}$ we conclude that%
\[
\lambda_{q}\leq\liminf_{s\rightarrow q^{-}}\lambda_{s}\leq\limsup
_{s\rightarrow q^{-}}\lambda_{s}\leq\mathcal{R}_{q}(u)
\]
for each $u\in C_{c}^{\infty}(\Omega)\setminus\{0\}.$ Since $C_{c}^{\infty
}(\Omega)$ is dense in $W_{0}^{1,p}(\Omega)$ this clearly implies that
\[
\lambda_{q}\leq\liminf_{s\rightarrow q^{-}}\lambda_{s}\leq\limsup
_{s\rightarrow q^{-}}\lambda_{s}\leq\mathcal{R}_{q}(u)\text{ \ for all }u\in
W_{0}^{1,p}(\Omega)\setminus\{0\}.
\]
Therefore,
\[
\lambda_{q}\leq\liminf_{s\rightarrow q^{-}}\lambda_{s}\leq\limsup
_{s\rightarrow q^{-}}\lambda_{s}\leq\lambda_{q}%
\]
from what follows that $\lim\limits_{s\rightarrow q^{-}}\lambda_{s}%
=\lambda_{q}.$
\end{proof}

\section{Bounds for $w_{q}$}

In this section we deduce some bounds for the extremal $w_{q}$ defined by
(\ref{wq}). Our results are based on level set techniques and inspired by
\cite{Bandle} and \cite{Ladyz}.

\begin{proposition}
\label{linfty}Let $1\leq q<p^{\star}$ and $\sigma\geq1.$ Then, it holds
\begin{equation}
2^{-\frac{N(p-1)+\sigma p}{p}}C_{q}\left\Vert w_{q}\right\Vert _{\infty
}^{\frac{N(p-q)+\sigma p}{p}}\leq\left\Vert w_{q}\right\Vert _{\sigma}%
^{\sigma} \label{estim2}%
\end{equation}
where
\begin{equation}
C_{q}:=\left(  \frac{p}{p+N(p-1)}\right)  ^{N+1}\left(  \frac{\mathcal{S}%
_{N,p}^{p}}{\lambda_{q}}\right)  ^{\frac{N}{p}}. \label{Cq}%
\end{equation}

\end{proposition}

\begin{proof}
Since $w_{q}$ is a positive weak solution of (\ref{diricq}) we have that
\begin{equation}
\int_{\Omega}\left\vert \nabla w_{q}\right\vert ^{q-2}\nabla w_{q}\cdot
\nabla\phi dx=\lambda_{q}\int_{\Omega}w_{q}^{q-1}\phi dx \label{LE}%
\end{equation}
for all test function $\phi\in W_{0}^{1,p}(\Omega)$.

For each $0<t<\left\Vert w_{q}\right\Vert _{\infty}$, define $A_{t}=\left\{
x\in\Omega:w_{q}>t\right\}  $. Since $w_{q}\in C^{1,\alpha}\left(
\overline{\Omega}\right)  $ for some $0<\alpha<1$ it follows that $A_{t}$ is
open and $\nabla(w_{q}-t)^{+}=\nabla w_{q}$ in $A_{t}.$

Thus, the function
\[
\left(  w_{q}-t\right)  ^{+}=\max\left\{  w_{q}-t,0\right\}  =\left\{
\begin{array}
[c]{ll}%
w_{q}-t, & \text{if \ }w_{q}>t\\
0, & \text{if \ }w_{q}\leq t
\end{array}
\right.
\]
belongs to $W_{0}^{1,p}\left(  \Omega\right)  $ and by using it as a test
function in (\ref{LE}) we obtain
\begin{equation}
\int_{A_{t}}\left\vert \nabla w_{q}\right\vert ^{p}dx=\lambda_{q}\int_{A_{t}%
}w_{q}^{q-1}\left(  w_{q}-t\right)  dx\leq\lambda_{q}\left\Vert w_{q}%
\right\Vert _{\infty}^{q-1}\left(  \left\Vert w_{q}\right\Vert _{\infty
}-t\right)  \left\vert A_{t}\right\vert . \label{x1}%
\end{equation}

Now, we estimate $\int_{A_{t}}\left\vert \nabla w_{q}\right\vert ^{p}dx$ from
below. Applying H\"{o}lder and Sobolev inequalities we obtain
\[
\left(  \int_{A_{t}}\left(  w_{q}-t\right)  dx\right)  ^{p}\leq\left(
\int_{A_{t}}\left(  w_{q}-t\right)  ^{p^{\star}}dx\right)  ^{\frac{p}%
{p^{\star}}}\left\vert A_{t}\right\vert ^{p-\frac{p}{p^{\star}}}%
\leq\mathcal{S}_{N,p}^{-p}\left\vert A_{t}\right\vert ^{p-\frac{p}{p^{\star}}%
}\int_{A_{t}}\left\vert \nabla w_{q}\right\vert ^{p}dx
\]
and thus,
\[
\mathcal{S}_{N,p}^{p}\left\vert A_{t}\right\vert ^{\frac{p}{p^{\star}}%
-p}\left(  \int_{A_{t}}\left(  w_{q}-t\right)  dx\right)  ^{p}\leq\int_{A_{t}%
}\left\vert \nabla w_{q}\right\vert ^{p}dx.
\]

By combining this inequality with (\ref{x1}) we obtain%
\[
\mathcal{S}_{N,p}^{p}\left\vert A_{t}\right\vert ^{\frac{p}{p^{\star}}%
-p}\left(  \int_{A_{t}}\left(  w_{q}-t\right)  dx\right)  ^{p}\leq\lambda
_{q}\left\Vert w_{q}\right\Vert _{\infty}^{q-1}(\left\Vert w_{q}\right\Vert
_{\infty}-t)\left\vert A_{t}\right\vert .
\]
Since $\frac{1}{p^{\star}}+\frac{1}{N}=\frac{1}{p}$ the previous inequality
can be rewritten as
\begin{equation}
\left(  \int_{A_{t}}\left(  w_{q}-t\right)  dx\right)  ^{\frac{N}{N+1}}%
\leq\left[  \lambda_{q}\mathcal{S}_{N,p}^{-p}\left\Vert w_{q}\right\Vert
_{\infty}^{q-1}(\left\Vert w_{q}\right\Vert _{\infty}-t)\right]  ^{\frac
{N}{p(N+1)}}\left\vert A_{t}\right\vert . \label{x2}%
\end{equation}

In the sequel we use twice the following Fubini's theorem: if $u\geq0$ is
measurable, $\sigma\geq1,$ and $E_{\tau}=\left\{  x:u(x)>\tau\right\}  ,$ then%
\begin{equation}
\int_{\Omega}u(x)^{\sigma}dx=\sigma\int_{0}^{\infty}\tau^{\sigma-1}\left\vert
E_{\tau}\right\vert d\tau. \label{cavalieri}%
\end{equation}

Let us define $g(t):=\int_{A_{t}}\left(  w_{q}-t\right)  dx$. It follows from
(\ref{cavalieri}) that%
\[
g(t)=\int_{0}^{\infty}\left\vert \left\{  w_{q}-t>\tau\right\}  \right\vert
d\tau=\int_{t}^{\infty}\left\vert \left\{  w_{q}>s\right\}  \right\vert
ds=\int_{t}^{\infty}\left\vert A_{s}\right\vert ds=\int_{t}^{\left\Vert
w_{q}\right\Vert _{\infty}}\left\vert A_{s}\right\vert ds
\]
and therefore $g^{\prime}\left(  t\right)  =-\left\vert A_{t}\right\vert
\leq0.$ Thus, (\ref{x2}) can be written as
\[
\left[  \lambda_{q}\mathcal{S}_{N,p}^{-p}\left\Vert w_{q}\right\Vert _{\infty
}^{q-1}(\left\Vert w_{q}\right\Vert _{\infty}-t)\right]  ^{-\frac{N}{p(N+1)}%
}\leq-g\left(  t\right)  ^{-\frac{N}{N+1}}g^{\prime}\left(  t\right)
\]
and integration over the interval $[t,\left\Vert w_{q}\right\Vert _{\infty}]$
produces%
\begin{equation}
C_{q}\left\Vert w_{q}\right\Vert _{\infty}^{-\frac{N(q-1)}{p}}\left(
\left\Vert w_{q}\right\Vert _{\infty}-t\right)  ^{\frac{N(p-1)+p}{p}}\leq
g\left(  t\right)  \label{x3}%
\end{equation}
where $C_{q}$ is given by (\ref{Cq}).

By using the fact that $g(t)\leq\left(  \left\Vert w_{q}\right\Vert _{\infty
}-t\right)  \left\vert A_{t}\right\vert $ we obtain from (\ref{x3}) that%
\[
C_{q}\left\Vert w_{q}\right\Vert _{\infty}^{-\frac{N(q-1)}{p}}\left(
\left\Vert w_{q}\right\Vert _{\infty}-t\right)  ^{\frac{N(p-1)}{p}}%
\leq\left\vert A_{t}\right\vert .
\]
If $\sigma\geq1$, multiplying the previous inequality by $\sigma t^{\sigma-1}$
and integrating over the interval $[0,\Vert w_{q}\Vert_{\infty}]$, we get%
\begin{equation}
C_{q}\left\Vert w_{q}\right\Vert _{\infty}^{-\frac{N(q-1)}{p}}\sigma\int
_{0}^{\left\Vert w_{q}\right\Vert _{\infty}}\left(  \left\Vert w_{q}%
\right\Vert _{\infty}-t\right)  ^{\frac{N(p-1)}{p}}t^{\sigma-1}dt\leq
\left\Vert w_{q}\right\Vert _{\sigma}^{\sigma} \label{x4}%
\end{equation}
since (\ref{cavalieri}) gives
\[
\left\Vert w_{q}\right\Vert _{\sigma}^{\sigma}=\int_{\Omega}w_{q}^{\sigma
}dx=\sigma\int_{0}^{\left\Vert w_{q}\right\Vert _{\infty}}t^{\sigma
-1}\left\vert A_{t}\right\vert dt.
\]

The change of variable $t=\tau\left\Vert w_{q}\right\Vert _{\infty}$ produces%
\begin{equation}
\int_{0}^{\left\Vert w_{q}\right\Vert _{\infty}}\left(  \!\left\Vert
w_{q}\right\Vert _{\infty}-t\!\right)  ^{\frac{N(p-1)}{p}}t^{\sigma
-1}dt=\left\Vert w_{q}\right\Vert _{\infty}^{\frac{N(p-1)+\sigma p}{p}}%
\int_{0}^{1}\left(  1-\tau\right)  ^{\frac{N(p-1)}{p}}\tau^{\sigma-1}d\tau.
\label{x5}%
\end{equation}
Since we have
\[
\int_{0}^{1}\left(  1-\tau\right)  ^{\frac{N(p-1)}{p}}\tau^{\sigma-1}d\tau
\geq(1/2)^{\frac{N(p-1)}{p}}\int_{0}^{\frac{1}{2}}\tau^{\sigma-1}d\tau
=\frac{2^{-\frac{N(p-1)+\sigma p}{p}}}{\sigma},
\]
combining this with (\ref{x4}) and (\ref{x5}) produces (\ref{estim2}).
\end{proof}

\medskip

We remark from Theorem \ref{T1} that
\begin{equation}
\lim_{q\rightarrow p^{\star}}C_{q}=\left(  \frac{p}{p+N(p-1)}\right)
^{N+1}<1. \label{Cqp*}%
\end{equation}
Thus it follows from the monotonicity of the function $q\mapsto\left\vert
\Omega\right\vert ^{\frac{p}{q}}\lambda_{q}$ that both $C_{q}$ and
$(C_{q})^{-1}$ are bounded.

\begin{corollary}
If $1\leq q\leq p^{\star}-\epsilon,$ then
\begin{equation}
\left\vert \Omega\right\vert ^{-\frac{1}{q}}\leq\left\Vert w_{q}\right\Vert
_{\infty}\leq\mathcal{C}_{\epsilon} \label{linfbounds}%
\end{equation}
where $\mathcal{C}_{\epsilon}$ is a positive constant that depends on
$\epsilon$ but not on $q.$
\end{corollary}

\begin{proof}
The first inequality is trivial, since $\Vert w_{q}\Vert_{q}=1.$ Let us
suppose that $1\leq q\leq p.$ It follows from (\ref{estim2}) with $\sigma=1$
that%
\[
2^{-\frac{N(p-1)+p}{p}}C_{q}\left\Vert w_{q}\right\Vert _{\infty}%
^{\frac{N(p-q)+p}{p}}\leq\left\Vert w_{q}\right\Vert _{1}\leq\left\vert
\Omega\right\vert ^{\frac{q-1}{q}}\left\Vert w_{q}\right\Vert _{q}=\left\vert
\Omega\right\vert ^{\frac{q-1}{q}}.
\]
Thus,%
\[
\left\Vert w_{q}\right\Vert _{\infty}\leq\widetilde{\mathcal{A}}:=\max_{1\leq
q\leq p}\left(  \frac{2^{\frac{N(p-1)+p}{p}}\left\vert \Omega\right\vert
^{\frac{q-1}{q}}}{C_{q}}\right)  ^{\frac{p}{p+N(p-q)}}.
\]

Now, let us consider $p\leq q\leq p^{\star}-\epsilon.$ Then, by making
$\sigma=q$ in (\ref{estim2}) we obtain
\[
2^{-\frac{N(p-1)+qp}{p}}C_{q}\left\Vert w_{q}\right\Vert _{\infty}%
^{\frac{N(p-q)+qp}{p}}\leq\left\Vert w_{q}\right\Vert _{q}^{q}=1.
\]
that is,
\[
\left\Vert w_{q}\right\Vert _{\infty}\leq\widetilde{\mathcal{B}}_{\epsilon
}:=\max_{p\leq q\leq p^{\star}-\epsilon}\left[  \frac{2^{\frac{N(p-1)+qp}{p}}%
}{C_{q}}\right]  ^{\frac{p}{(N-p)(p^{\star}-q)}},
\]
since $N(p-q)+qp=(N-p)(p^{\star}-q).$

Therefore, $\ \left\Vert w_{q}\right\Vert _{\infty}\leq\max\left\{
\widetilde{\mathcal{A}},\widetilde{\mathcal{B}}_{\epsilon}\right\}
:=\mathcal{C}_{\epsilon}.$
\end{proof}

Note from (\ref{Cqp*}) that $\widetilde{\mathcal{B}}_{\epsilon}\rightarrow
\infty$ as $\epsilon\rightarrow0^{+}.$

\section{Absolute continuity}

In this section we prove our main result: the absolute continuity of the
function $q\mapsto\lambda_{q}$ in the closed interval $[1,p^{\star}].$ For
this we first prove the Lipschitz continuity of the function $q\mapsto
\left\vert \Omega\right\vert ^{\frac{p}{q}}\lambda_{q}$ in each close interval
of the form $[1,p^{\star}-\epsilon].$ Obviously, this is equivalent to the
Lipschitz continuity of the function $q\mapsto\lambda_{q}$ in same interval.

\begin{theorem}
\label{main}For each $\epsilon>0$, there exists a positive constant
$\mathcal{L}_{\epsilon}$ such that%
\[
\left\vert \left\vert \Omega\right\vert ^{\frac{p}{s}}\lambda_{s}-\left\vert
\Omega\right\vert ^{\frac{p}{q}}\lambda_{q}\right\vert \leq\mathcal{L}%
_{\epsilon}\left\vert s-q\right\vert
\]
for all $s,q\in\lbrack1,p^{\star}-\epsilon].$
\end{theorem}

\begin{proof}
Without loss of generality let us suppose that $s<q.$ Thus, the monotonicity
of the function $\tau\mapsto\left\vert \Omega\right\vert ^{\frac{p}{\tau}%
}\lambda_{\tau}$ implies%
\[
\left\vert \left\vert \Omega\right\vert ^{\frac{p}{s}}\lambda_{s}-\left\vert
\Omega\right\vert ^{\frac{p}{q}}\lambda_{q}\right\vert =\left\vert
\Omega\right\vert ^{\frac{p}{s}}\lambda_{s}-\left\vert \Omega\right\vert
^{\frac{p}{q}}\lambda_{q}.
\]

Take $t\in\mathbb{R}$ so that $s\leq t\leq q.$ It follows from (\ref{estim2})
with $\sigma=1$ that%
\[
2^{-\frac{N(p-1)+p}{p}}C_{q}\left\Vert w_{q}\right\Vert _{\infty}%
^{1+\frac{N(p-q)}{p}}\leq\left\Vert w_{q}\right\Vert _{1}\leq\left\vert
\Omega\right\vert ^{1-\frac{1}{t}}\left\Vert w_{q}\right\Vert _{t}%
\]
and therefore%
\[
\frac{\left\vert \Omega\right\vert ^{\frac{1}{t}}\left\Vert w_{q}\right\Vert
_{\infty}}{\left\Vert w_{q}\right\Vert _{t}}\leq\frac{2^{\frac{N(p-1)+p}{p}%
}\left\vert \Omega\right\vert }{C_{q}\left\Vert w_{q}\right\Vert _{\infty
}^{\frac{N(p-q)}{p}}}.
\]

Hence, for $1\leq q\leq p$ the first inequality in (\ref{linfbounds}) gives%
\[
\frac{\left\vert \Omega\right\vert ^{\frac{1}{t}}\left\Vert w_{q}\right\Vert
_{\infty}}{\left\Vert w_{q}\right\Vert _{t}}\leq\frac{2^{\frac{N(p-1)+p}{p}%
}\left\vert \Omega\right\vert }{C_{q}\left\Vert w_{q}\right\Vert _{\infty
}^{\frac{N(p-q)}{p}}}\leq\mathcal{A}:=2^{\frac{N(p-1)+p}{p}}\max_{1\leq q\leq
p}\frac{\left\vert \Omega\right\vert ^{1+\frac{N(p-q)}{pq}}}{C_{q}}%
\]
while for $p\leq q\leq p^{\star}-\epsilon$ the second inequality in
(\ref{linfbounds}) gives%
\[
\frac{\left\vert \Omega\right\vert ^{\frac{1}{t}}\left\Vert w_{q}\right\Vert
_{\infty}}{\left\Vert w_{q}\right\Vert _{t}}\leq\frac{2^{\frac{N(p-1)+p}{p}%
}\left\vert \Omega\right\vert }{C_{q}}\left\Vert w_{q}\right\Vert _{\infty
}^{\frac{N(q-p)}{p}}\leq\mathcal{B}_{\epsilon}:=2^{\frac{N(p-1)+p}{p}%
}\left\vert \Omega\right\vert \max_{p\leq q\leq p^{\star}-\epsilon}%
\frac{\mathcal{C}_{\epsilon}^{\frac{N(q-p)}{p}}}{C_{q}}.
\]
Therefore,
\begin{equation}
\frac{\left\vert \Omega\right\vert ^{\frac{1}{t}}\left\Vert w_{q}\right\Vert
_{\infty}}{\left\Vert w_{q}\right\Vert _{t}}\leq\mathcal{D}_{\epsilon}%
:=\max\left\{  \mathcal{A},\mathcal{B}_{\epsilon}\right\}  . \label{Depsilon}%
\end{equation}

Thus,
\begin{align*}
K(t,w_{q})  &  =\frac{t\int_{\Omega}\left\vert w_{q}\right\vert ^{t}%
\ln\left\vert w_{q}\right\vert dx}{\left\Vert w_{q}\right\Vert _{t}^{t}}%
+t\ln\left(  \frac{\left\vert \Omega\right\vert ^{\frac{1}{t}}}{\left\Vert
w_{q}\right\Vert _{t}}\right) \\
&  \leq\frac{t(\ln\left\Vert w_{q}\right\Vert _{\infty})\int_{\Omega
}\left\vert w_{q}\right\vert ^{t}dx}{\left\Vert w_{q}\right\Vert _{t}^{t}%
}+t\ln\left(  \frac{\left\vert \Omega\right\vert ^{\frac{1}{t}}}{\left\Vert
w_{q}\right\Vert _{t}}\right)  =t\ln\left(  \frac{\left\vert \Omega\right\vert
^{\frac{1}{t}}\left\Vert w_{q}\right\Vert _{\infty}}{\left\Vert w_{q}%
\right\Vert _{t}}\right)  \leq t\mathcal{D}_{\epsilon}%
\end{align*}
and we obtain
\[
\exp\left(  p\int_{s}^{q}\frac{K(t,w_{q})}{t^{2}}dt\right)  \leq\exp\left(
p\mathcal{D}_{\epsilon}\int_{s}^{q}\frac{dt}{t}\right)  =\left(  \frac{q}%
{s}\right)  ^{p\mathcal{D}_{\epsilon}}.
\]

But
\begin{align*}
\left\vert \Omega\right\vert ^{\frac{p}{s}}\lambda_{s}  &  \leq\left\vert
\Omega\right\vert ^{\frac{p}{s}}\mathcal{R}(w_{q})\\
&  =\left\vert \Omega\right\vert ^{\frac{p}{q}}\mathcal{R}(w_{q})\exp\left(
p\int_{s}^{q}\frac{K(t,w_{q})}{t^{2}}dt\right)  =\left\vert \Omega\right\vert
^{\frac{p}{q}}\lambda_{q}\exp\left(  p\int_{s}^{q}\frac{K(t,w_{q})}{t^{2}%
}dt\right)
\end{align*}
yields
\begin{align*}
\left\vert \Omega\right\vert ^{\frac{p}{s}}\lambda_{s}-\left\vert
\Omega\right\vert ^{\frac{p}{q}}\lambda_{q}  &  \leq\left\vert \Omega
\right\vert ^{\frac{p}{q}}\lambda_{q}\left[  \exp\left(  p\int_{s}^{q}%
\frac{K(t,w_{q})}{t^{2}}dt\right)  -1\right] \\
&  \leq\left\vert \Omega\right\vert ^{p}\lambda_{1}\left[  \exp\left(
p\int_{s}^{q}\frac{K(t,w_{q})}{t^{2}}dt\right)  -1\right]  \leq\left\vert
\Omega\right\vert ^{p}\lambda_{1}\left[  \left(  \frac{q}{s}\right)
^{p\mathcal{D}_{\epsilon}}-1\right]  .
\end{align*}

Therefore,%
\[
0<\frac{\left\vert \Omega\right\vert ^{\frac{p}{s}}\lambda_{s}-\left\vert
\Omega\right\vert ^{\frac{p}{q}}\lambda_{q}}{q-s}\leq\frac{\left\vert
\Omega\right\vert ^{p}\lambda_{1}}{s}\frac{\left(  \frac{q}{s}\right)
^{p\mathcal{D}_{\epsilon}}-1}{\frac{q}{s}-1}\leq\left\vert \Omega\right\vert
^{p}\lambda_{1}H(\frac{q}{s})
\]
where
\[
H(\xi)=\frac{\xi^{p\mathcal{D}_{\epsilon}}-1}{\xi-1};\ \ \ \ \ \ \ \ \ \ 1\leq
\xi\leq p^{\star}-\epsilon.
\]

Since $\lim\limits_{\xi\rightarrow1^{+}}H(\xi)=p\mathcal{D}_{\epsilon}$, we
conclude that $H$ is bounded in $[1,p^{\star}-\epsilon]$ and thus%
\[
\frac{\exp\left(  p\int_{s}^{q}\frac{K(t,w_{q})}{t^{2}}dt\right)  -1}{q-s}%
\leq\mathcal{L}_{\epsilon}:=\max_{1\leq\xi\leq p^{\star}-\epsilon}H(\xi).
\]
Hence, for $1\leq s<q\leq p^{\star}-\epsilon$ we have
\[
\left\vert \left\vert \Omega\right\vert ^{\frac{p}{s}}\lambda_{s}-\left\vert
\Omega\right\vert ^{\frac{p}{q}}\lambda_{q}\right\vert =\left\vert
\Omega\right\vert ^{\frac{p}{s}}\lambda_{s}-\left\vert \Omega\right\vert
^{\frac{p}{q}}\lambda_{q}\leq\mathcal{L}_{\epsilon}\left(  q-s\right)
=\mathcal{L}_{\epsilon}\left\vert s-q\right\vert .\vspace*{-0.7cm}%
\]

\end{proof}

\begin{theorem}
The function $q\mapsto\lambda_{q}$ is absolutely continuous in $\left[
1,p^{\star}\right]  .$
\end{theorem}

\begin{proof}
According to Corollary \ref{bv} the function $q\mapsto\lambda_{q}$ is of
bounded variation. Therefore, its derivative $(\lambda_{q})^{\prime}$ exists
almost everywhere in $\left[  1,p^{\star}\right]  $ and it is Lebesgue
integrable in this interval. Thus, Lebesgue's dominated convergence theorem
implies that
\begin{equation}
\lim_{q\rightarrow p^{\star}}\int_{1}^{q}(\lambda_{s})^{\prime}ds=\int
_{1}^{p^{\star}}(\lambda_{s})^{\prime}ds. \label{dominated}%
\end{equation}

On the other hand, since Lipschtiz continuity implies absolute continuity it
follows from Theorem \ref{main} that $\lambda_{q}$ is absolutely continuous in
each interval of the form $[1,p^{\star}-\epsilon].$ Therefore,
\begin{equation}
\lambda_{q}=\lambda_{1}+\int_{1}^{q}(\lambda_{s})^{\prime}ds,\text{ \ for
}1\leq q<p^{\star}. \label{integ}%
\end{equation}
Hence, the left-continuity (\ref{leftp*}) combined with (\ref{dominated})
imply that (\ref{integ}) is also valid for $q=p^{\star}.$ We have concluded
that $\lambda_{q}$ is the indefinite integral of a Lebesgue integrable (its
derivative) function what guarantees that $\lambda_{q}$ is absolutely continuous.
\end{proof}

\end{document}